\documentclass[12pt]{Paper}

 \usepackage[foot]{amsaddr}
\usepackage[framemethod=TikZ]{mdframed}
\usepackage{lipsum}

\usepackage{fancybox}
 \usepackage{enumerate}
  
\usepackage{cite}
\usepackage{scalerel}
\usepackage{amsthm,lipsum}
\usepackage{amsmath}
\usepackage{amssymb}
  \usepackage{paralist}
  \usepackage{graphics} 
  \usepackage{epsfig} 
  \usepackage{csquotes}
  \usepackage{subcaption}
\usepackage{graphicx}  
\usepackage{epstopdf}
 \usepackage[colorlinks=true]{hyperref}
\usepackage{bm}
\usepackage{siunitx}
\hypersetup{urlcolor=blue, citecolor=red}
\usepackage[margin=1.2in]{geometry}
\AfterEndEnvironment{figure}{\noindent\ignorespaces}

\newtheorem{theorem}{Theorem}[section]

\newtheorem{lemma}[theorem]{Lemma}

\newtheorem{conjecture}[theorem]{Conjecture}

\theoremstyle{definition}
\newtheorem{remark}{Remark}[section]

\usepackage[T1]{fontenc}    
\usepackage[utf8]{inputenc} 

\renewenvironment{proof}{{\bf \noindent Proof.}}{\hfill $\square$\vspace{0.25cm}} 

\def\S{\mathbb{S}}
\def\ra{\frac{r}{\sqrt{3}}}

\def\f{\longrightarrow}

\def\N{\mathbb{N}}

\def\l{\lambda}

\def\a{\alpha}
\def\b{\beta}
\def\<{\langle}
\def\>{\rangle}

\def\L{\mathcal{L}}

\def\R{\mathbb{R}}

\def\O{\mathcal{O}}
\def\inte{\textnormal{int}\,}
\def\clo{\textnormal{cl}\,}

\def\bdry{\textnormal{bdry}\,}

\def\proj{\textnormal{proj}\,}

\def\nvarphi{\bp\nvarphi}

\def\bp{\hspace{-0.08cm}}

\def\bbbp{\hspace{-0.01cm}}

\def\proj{\textnormal{proj}}

\newenvironment{proofof}[1]
{\par\noindent\textbf{Proof of #1.}\ }
{\hfill$\square$\par}

\makeatletter
\newcommand{\Largeish}{\fontsize{14.5}{16.5}\selectfont} 
\makeatother

\makeatletter
\newcommand{\Largeishsub}{\fontsize{13.5}{15.5}\selectfont} 
\makeatother

\makeatletter
\renewcommand{\section}{\@startsection{section}{1}{0pt}%
  {1em}{1em}{\normalfont\Largeish\bfseries}}
\makeatother

\usepackage{titlesec}

\titleformat{\subsection}[block] {\normalfont\Largeishsub\bfseries}{\thesubsection}{1em}{}

\usepackage{parskip}

\makeatletter
\newcommand*\rel@kern[1]{\kern#1\dimexpr\macc@kerna}
\newcommand*\widebar[1]{%
  \begingroup
  \def\mathaccent##1##2{%
    \rel@kern{0.8}%
    \overline{\rel@kern{-0.8}\macc@nucleus\rel@kern{0.2}}%
    \rel@kern{-0.2}%
  }%
  \macc@depth\@ne
  \let\math@bgroup\@empty \let\math@egroup\macc@set@skewchar
  \mathsurround\z@ \frozen@everymath{\mathgroup\macc@group\relax}%
  \macc@set@skewchar\relax
  \let\mathaccentV\macc@nested@a
  \macc@nested@a\relax111{#1}%
  \endgroup
}
\makeatother

\begin{document}

\title[Validity of the Strong Version\;$\dots$]{Validity of the Strong Version of the Union of Uniform  Closed Balls Conjecture in the Plane}

\author[C. Nour and J. Takche]{}
\author{Chadi Nour$^*$}
\address{Department of Computer Science and Mathematics, Lebanese American University, Byblos Campus, P.O. Box 36, Byblos, Lebanon}
\email{cnour@lau.edu.lb}
\thanks{$^*$Corresponding author}

\author{Jean Takche}
\email{jtakchi@lau.edu.lb}




\begin{abstract} We prove the validity of the strong version of the union of uniform closed balls conjecture, formulated in 2011 as \cite[Conjecture 2.5]{JCA2011}, in the plane.
\end{abstract}
\maketitle
\vspace{-0.7cm}
\section{Introduction} Let $S\subset\mathbb{R}^n$ be a nonempty and closed set and let $r>0$. We recall that $S$ is said to satisfy the {\it interior $r$-sphere condition} if, for every {\it boundary} point of $S$, there exists a closed ball of radius $r$ contained in $S$ whose boundary contains that point. This is equivalent to saying that $S$ is regular closed and for every boundary point $s$ of $S'$, the complement of the interior of $S$, there exists a nonzero proximal normal vector to $S'$ at $s$ which is {\it realized by an $r$-sphere}. Here, for a boundary point $a$ of a nonempty and closed set $A\subset\R^n$, we say that a nonzero vector $\zeta\in N_{A}^P(a)$, the {\it proximal normal cone} to $A$ at $a$, is realized by an $r$-sphere, if \begin{equation*} \label{realized}\left\<\frac{\zeta}{\|\zeta\|},x-a\right\>\leq \frac{1}{2r}\|x-a\|^2,\;\forall x\in A\;\;\;\; \left[\hbox{or equivalently}\; B\left(a+r\frac{\zeta}{\|\zeta\|};r\right)\cap A=\emptyset\right],\end{equation*} where $B(x;\rho)$ denotes the open ball of radius $\rho$ centered at $x$. One can easily see that if $S$ is the union of closed balls with common radius $r$, then it satisfies the interior $r$-sphere condition. The converse is not necessarily true as it is shown in \cite[Example 4.1]{JCA2009}. In this latter, Nour, Stern and Takche provided a nonempty and closed set $S\subset\R^2$ satisfying the interior $1$-sphere condition but fails to be the union of closed balls with common radius $1$. While the set $S$ of \cite[Example 4.1]{JCA2009} is not the union of closed balls with common radius $1$, it is the union of closed balls with common radius $\frac{1}{\sqrt{3}}$. This led Nour, Stern and Takche to introduce in  \cite[Conjecture 4.4]{JCA2009} the following conjecture called {\it weak version of the union of uniform closed balls conjecture}. 

\begin{conjecture}[\bbbp{\cite[Conjecture 4.4]{JCA2009}}] \label{weakconj} Let $S\subset\mathbb{R}^n$ be a nonempty and closed set satisfying the interior $r$-sphere condition for some $r>0$. Then there exists $r'>0$ such that $S$ is the union of closed balls with common radius $r'$.
\end{conjecture}
In \cite{JCA2011}, Nour, Stern and Takche proved, using techniques from nonsmooth analysis,  the validity of Conjecture \ref{weakconj} for $r'=\frac{r}{2}$. In fact, they proved the following theorem. 

\begin{theorem}[\bbbp{\cite{JCA2011}}] \label{weakth} Let $S\subset\mathbb{R}^n$ be a nonempty and closed set satisfying the interior $r$-sphere condition for some $r>0$. Then $S$ is the union of closed balls with common radius $\frac{r}{2}$. \end{theorem}

In \cite[Example 2]{DCDS2011}, Nour, Stern and Takche generalized  \cite[Example 4.1]{JCA2009} to $\R^n$. More precisely, they provided, for any $r>0$, a nonempty and closed set $S\subset\R^n$ satisfying the interior $r$-sphere condition but fails to be the union of closed balls with common radius $r'$ for any $r'> \frac{nr}{2\sqrt{n^2-1}}$. As a conclusion and since $$\frac{r}{2}<  \frac{nr}{2\sqrt{n^2-1}}\xrightarrow[n\f+\infty]{}\frac{r}{2},$$ Nour, Stern and Takche deduced that $\frac{r}{2}$ is the {\it largest} radius $r'$ that can be taken in Conjecture \ref{weakconj} and works in $\R^n$ for all $n\geq 2$. Moreover, since no counterexample in which $S$ fails to be the union of closed balls with common radius  $\frac{nr}{2\sqrt{n^2-1}}$ is found, they introduced in  \cite[Conjecture 3]{DCDS2011}, see also \cite[Conjecure 2.5]{JCA2011}, the following conjecture called {\it strong version of the union of uniform closed balls conjecture}. 

\begin{conjecture}[\bbbp{\cite[Conjecture 2.5]{JCA2011}}] \label{strongconj} Let $S\subset\mathbb{R}^n$ be a nonempty and closed set satisfying the interior $r$-sphere condition for some $r>0$. Then $S$ is the union of closed balls with common radius $\frac{nr}{2\sqrt{n^2-1}}$. 
\end{conjecture}

Although Conjecture~\ref{strongconj} was formulated more than fifteen years ago, it remains unresolved even in the lowest nontrivial dimensions. In particular, neither a proof nor a counterexample is known in the planar case $n=2$ or in the three-dimensional case $n=3$.

The aim of the present paper is to resolve Conjecture~\ref{strongconj} in the planar case. Our main result, stated in the following theorem, shows that the strong version of the union of uniform closed balls conjecture holds in $\mathbb{R}^2$ with the optimal constant $\frac{r}{\sqrt{3}}$.

\begin{theorem} \label{mainthm}
Let $S\subset\mathbb{R}^2$ be a nonempty and closed set satisfying the interior $r$-sphere condition for some $r>0$. Then $S$ is the union of closed balls with common radius $\frac{r}{\sqrt{3}}$. 
\end{theorem}

The proof of Theorem~\ref{mainthm} relies on a geometric argument that is specific to the planar setting. 
The key ingredient is a local analysis of inward normal directions at boundary points under the interior sphere condition. 
In dimension two, proximal normal cones can be represented by sectors of the unit circle, which allows the directions to be ordered and compared via oriented angles. This structure makes it possible to identify extremal normal directions and to derive sharp angular estimates for the corresponding interior tangent balls. Such an argument relies essentially on the one-dimensional geometry of directions and does not admit a direct analogue in higher dimensions.

The paper is organized as follows. In Section~\ref{prem}, we introduce the notation used throughout the paper and recall basic definitions and preliminary results related to proximal analysis and the interior sphere condition. Section~\ref{proof} is devoted to the proof of Theorem~\ref{mainthm}. In the final section, Section \ref{sectionfinal}, we discuss possible extensions of the result and outline perspectives for a generalization to higher-dimensional spaces.

%

\section{Preliminaries} \label{prem}

Throughout the paper, $\langle\cdot,\cdot\rangle$ denotes the Euclidean inner
product in $\mathbb{R}^n$ and $\|\cdot\|$ the associated norm. For $x\in\mathbb{R}^n$ and $r>0$, $\overline{B}(x;r)$, $B(x;r)$, and $\mathbb{S}(x;r)$ denote the closed ball, the open ball, and the sphere centered at $x$ with radius $r$, respectively.  For a set $A\subset\mathbb{R}^n$, the closure, interior, boundary, and complement of $A$ are denoted by $\clo A$, $\inte A$,
$\bdry A$, and $A^c$, respectively. Moreover, we set
\[
A':=(\inte A)^c=\clo(A^c).
\]

The distance from a point $x$ to a nonempty and closed set $A\subset\mathbb{R}^n$
is denoted by
\[
d(x,A):=\inf_{a\in A}\|x-a\|.
\]
We denote by $\operatorname{proj}_A(x)$ the set of points in $A$ closest to $x$,
that is,
\[
\operatorname{proj}_A(x):=\{a\in A:\|x-a\|=d(x,A)\}.
\]
In the planar case, for two nonzero vectors $u,v\in\mathbb{R}^2$, we denote by
$\angle(u,v)\in[0,2\pi[$ the unique counterclockwise oriented angle from $u$
to $v$. 
For three distinct points $a,b,c\in\mathbb{R}^2$, we denote by $\angle abc$
the angle at $b$ from $a$ to $c$, that is,
\[
\angle abc := \angle(a-b,\,c-b).
\]

We now recall several notions from nonsmooth analysis that will be used
throughout the paper. Comprehensive accounts of these concepts can be found in
the monographs \cite{clsw,mord,penot,rockwet,thibault}. Let $A\subset\mathbb{R}^n$
be a nonempty and closed set and let $a\in\bdry A$.
The proximal normal cone to $A$ at $a$ is defined by
\[
N_A^P(a):=\left\{\zeta\in\mathbb{R}^n:\;
\exists \sigma>0 \text{ such that }
\langle \zeta, x-a\rangle \le \sigma \|x-a\|^2,\;\; \forall x\in A \right\}.
\]
A nonzero vector $\zeta\in N_A^P(a)$ is said to be realized by an $r$-sphere if
\[
\left\langle \frac{\zeta}{\|\zeta\|},x-a\right\rangle
\le \frac{1}{2r}\|x-a\|^2,\;\; \forall x\in A,
\]
or equivalently if
\[
B\left(a+r\frac{\zeta}{\|\zeta\|};r\right)\cap A=\emptyset.
\]
Note that for $y\not\in A$, if $a\in\proj_A(y)=\proj_{\bdry A}(y)$ then $y-a\in N_A^P(a)$ and is realized by a $\|y-a\|$-sphere, that is, \[
\left\langle y-a,x-a\right\rangle
\le \frac{1}{2}\|x-a\|^2,\;\; \forall x\in A.
\]
A boundary point $a$ of $A$ is said to be {\it regular} if the proximal
normal cone $N_A^P(a)$ is a half-line, that is,  there exists a unit vector $\xi_a\in \R^n$ such that $$N_A(a)=\{\l\xi_a : \l\geq 0\}. $$  
Recall that a nonempty and closed set $A\subset\mathbb{R}^n$ satisfies the
interior $r$-sphere condition, for some $r>0$, if for every $a\in\bdry A$ there
exists $x\in A$ such that
\[
a\in\overline B(x;r)\subset A.
\]
The following elementary observation will be used repeatedly in the sequel.

\begin{lemma}\label{lem0}
Let $A\subset\mathbb{R}^n$ be nonempty and closed and let $r>0$. Then $A$ satisfies
the interior $r$-sphere condition if and only if $A$ is regular closed and for
any $a\in\bdry A$, there exists at least one unit proximal normal
$\zeta\in N_{A'}^P(a)$ which is realized by an $r$-sphere.
\end{lemma}
\begin{proof}
Assume first that $A$ satisfies the interior $r$-sphere condition. Then for
every $a\in\bdry A$ there exists $x\in A$ such that
$a\in\overline B(x;r)\subset A$, which implies that $a\in\clo(\inte A)$.
Therefore $\bdry A\subset \clo(\inte A)$ and hence $A=\clo(\inte A)$, that is,
$A$ is regular closed. Moreover, the existence of such a ball immediately
yields a unit proximal normal to $A'$ at $a$ realized by an $r$-sphere.

Conversely, if $A$ is regular closed and for every $a\in\bdry A$ there exists a
unit proximal normal $\zeta\in N_{A'}^P(a)$ realized by an $r$-sphere, then by
definition
\[
B(a+r\zeta;r)\cap A'=\emptyset,
\]
which implies $\overline B(a+r\zeta;r)\subset A$, and hence $A$ satisfies the
interior $r$-sphere condition.
\end{proof}

\section{Proof of  Theorem~\ref{mainthm}} \label{proof}  Before presenting the details of the proof, we establish two preliminary
results. We begin with the following planar geometric lemma, which will be used in the final step of the argument.

\begin{lemma} \label{geolem} Let $r>0$, $r_0>0$, $\a\in]0,\pi[$, and $\b\in [0,\pi[$ such that $$0\leq \b\leq\a\leq \pi-\b\;\;\;\hbox{and}\;\;\;r_0\sin\a\geq r\sin\b\geq0.$$
We consider in $\R^2$ the three points $$A:=(r_0\cos\a,r_0\sin\a),\;\;B:=( r\cos\b,r\sin\b),\,\,\hbox{and}\;\;C:=( -r\cos\b,r\sin\b).$$
Denote by $S_A$, $S_B$, and $S_C$ the circles centered at $A$, $B$, and $C$
with radii $r_0$, $r$, and $r$, respectively. Let $O$ denote the origin and define the points $D$, $E$, and $F$ by $$S_A\cap S_C=\{O,D\},\;\;S_A\cap S_B=\{O,E\},\;\;\hbox{and}\;\;S_A\cap y\hbox{-axis}=\{O,F\}.$$
The following assertions hold:
\begin{enumerate}[$(i)$]
\item $\angle EOD=\pi- \angle CAB.$
\item If $r_0< \frac{r}{\sqrt{3}}$ then:
\begin{enumerate}[$(a)$] 
\item $\angle EOD<\frac{\pi}{3}$.
\item For $M$ and $N$ two points on the circle $S_A$ that lie outside both circles $S_B$ and $S_C$, we have  $\angle NOM<\frac{\pi}{3}$.
\end{enumerate}
\end{enumerate}
\end{lemma}
\begin{proof} The geometric configuration is shown in Figure~\ref{Fig1}.
\begin{figure}[h]
\centering
\includegraphics[scale=0.4]{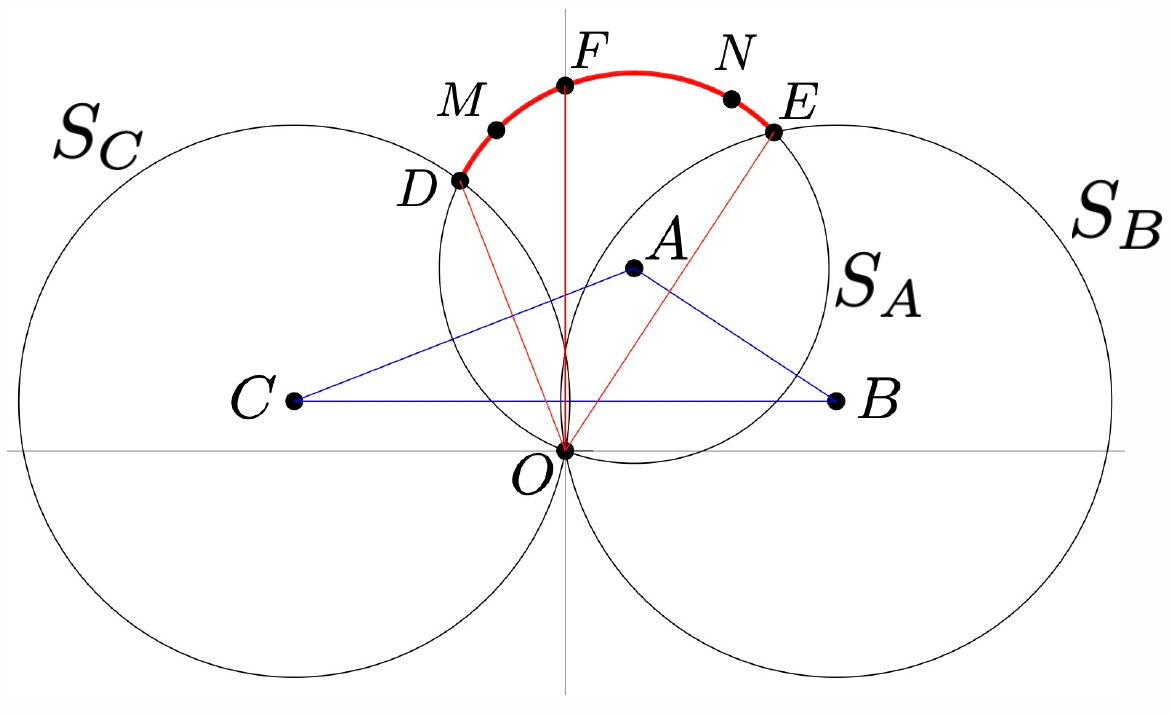}
\caption{\label{Fig1} Geometric configuration of Lemma~\ref{geolem}}
\end{figure}

$(i)$: $\angle EOD=\angle EOF+\angle FOD=\angle ABC+\angle BCA=\pi-\angle CAB.$\vspace{0.2cm}\\
$(ii)(a)$: Applying the law of cosines in triangle $ABC$, we obtain
\[
BC^2 = AB^2 + AC^2 - 2AB\,AC\cos(\angle CAB).
\]
We compute each term. First,
\[
BC^2 = 4r^2\cos^2\beta = 4r^2 - 4r^2\sin^2\beta.
\]
Next,
\begin{eqnarray*}
AB^2
&=& (r\cos\beta - r_0\cos\alpha)^2 + (r\sin\beta - r_0\sin\alpha)^2 \\
&=& r^2 + r_0^2 - 2rr_0\sin\beta\sin\alpha - 2rr_0\cos\beta\cos\alpha,
\end{eqnarray*}
and
\begin{eqnarray*}
AC^2
&=& (-r\cos\beta - r_0\cos\alpha)^2 + (r\sin\beta - r_0\sin\alpha)^2 \\
&=& r^2 + r_0^2 - 2rr_0\sin\beta\sin\alpha + 2rr_0\cos\beta\cos\alpha.
\end{eqnarray*}
Hence,
\begin{eqnarray*}
-\cos(\angle CAB)
&=& \frac{BC^2 - AB^2 - AC^2}{2AB.AC} \\ &=&
\frac{2r^2 - 2r_0^2 + 4rr_0\sin\beta\sin\alpha - 4r^2\sin^2\beta}
{2\sqrt{(r^2 + r_0^2 - 2rr_0\sin\beta\sin\alpha)^2
       - (2rr_0\cos\beta\cos\alpha)^2}}\\&\geq& \frac{2r^2 - 2r_0^2 + 4rr_0\sin\beta\sin\alpha - 4r^2\sin^2\beta}
{2\sqrt{(r^2 + r_0^2 - 2rr_0\sin\beta\sin\alpha)^2}}\\&=& \frac{2r^2 - 2r_0^2 + 4rr_0\sin\beta\sin\alpha - 4r^2\sin^2\beta}
{2(r^2 + r_0^2 - 2rr_0\sin\beta\sin\alpha)}.
\end{eqnarray*}
Since $r_0\sin\alpha \ge r\sin\beta$, we obtain
\[
2r^2 - 2r_0^2 + 4r\sin\beta(r_0\sin\alpha - r\sin\beta)\ge 2r^2 - 2r_0^2,
\]
and therefore,
\[
-\cos(\angle CAB)
\ge
\frac{2r^2 - 2r_0^2}
{2(r^2 + r_0^2 - 2rr_0\sin\beta\sin\alpha)}
\ge
\frac{r^2 - r_0^2}{r^2 + r_0^2}.
\]
Hence,
\[
\cos(\angle CAB)\le \frac{r_0^2-r^2}{r_0^2+r^2}
=1-\frac{2r^2}{r_0^2+r^2}.
\]
If $r_0< \dfrac{r}{\sqrt{3}}$, then
\[
\cos(\angle CAB)
<1-\frac{2r^2}{\frac{r^2}{3}+r^2}
=1-\frac{2}{\frac{1}{3}+1}
=-\frac12,
\]
so $\angle CAB>\dfrac{2\pi}{3}$. This yields, using $(i)$, that $\angle EOD=\pi-\angle CAB<\dfrac{\pi}{3}$.\vspace{0.2cm}\\
$(ii)(b)$: $\angle NOM\leq \angle EOD<\dfrac{\pi}{3}$.
\end{proof}

The next lemma will be used repeatedly in the proof of Theorem~\ref{mainthm}.

\begin{lemma}\label{secondlem} Let $S\subset\mathbb{R}^2$ be a nonempty and closed set satisfying the interior $r$-sphere condition for some $r>0$. For $r^*\in ]0,r[$, let $S_{r^*}$ denote the union of all the closed balls of radius $r^*$ contained in $S$. Let $s_0\in \bdry S$ and let  $x_0\in S\cap (S_{r^*})^c$. Set $r_0:=\|x_0-s_0\|$ and $\zeta:=\frac{x_0-s_0}{\|x_0-s_0\|}$.\footnote{Note that $x_0\not=s_0$ since $\bdry S\subset S_{r^*}$} Assume $\zeta_{s_0}\in N_{S'}^P(s_0)$ is a unit normal vector to $S'$ at $s_0$ realized by an $r'$-sphere for some $r'<r^*$ and that $x_0\in \overline{B}(s_0+r'\zeta_{s_0},r')\subset S$. Then $s_0$ is not regular, and there exist a unit vector $u_x$ and  two distinct unit normal vectors $\zeta_0$ and $\xi_0$ to $S'$ at $s_0$, realized by an $r$-sphere, such that:
\begin{enumerate}[$(i)$]
\item $\pi\leq \angle(\zeta_0,\xi_0)\leq 2\pi-2\cos^{-1}\!\left(\frac{r_0}{2r}\right)$.
\item $\cos^{-1}\!\left(\frac{r_0}{2r}\right)\le \angle(\zeta,\zeta_0)\leq\angle(\zeta,u_x) - \frac{\pi}{2}$.
\item $\angle(\zeta,u_x) +\frac{\pi}{2}\leq\angle(\zeta,\xi_0)\leq 2\pi-\cos^{-1}\!\left(\frac{r_0}{2r}\right)$.
\end{enumerate}
\end{lemma} 

\begin{proof} Since $s_0\in\bdry S$, we have $\overline{B}\!\left(s_0,\frac{1}{n}\right)\not\subset S$ for all $n$. Let $x_n \in \overline{B}\!\left(s_0,\frac{1}{n}\right)\cap S^c$ and set $r_n:=\|x_n-s_0\|\le \frac{1}{n}< r_0$ for $n$ sufficiently large.
For $y_n:=s_0+r_n\zeta$ so that $\|y_n-s_0\|=r_n\le r_0$, we have for $n$ sufficiently large 
\[
\|y_n-s_0\|=r_n <r_0.
\]
This yields that for $n$ sufficiently large $$y_n\in ]s_0,x_0[\subset B(s_0+r'\zeta_{s_0};r')\subset \overline{B}(s_0+r'\zeta_{s_0};r')\subset S.$$
Hence for $n$ sufficiently large, $y_n\in\inte S$. Let $C_n$ be the circle centered at $s_0$ with radius $r_n$. The clockwise arc of $C_n$ joining $x_n\notin S$ to $y_n\in\inte S$ intersects $\bdry S$ at a point $s_n$, and the counterclockwise arc 
of $C_n$ joining $x_n$ to $y_n$ intersects $\bdry S$ at a point $s_n'$. Clearly, the open arc of $C_n$ joining $s_n$ to $s_n'$ and containing $x_n$ is contained in $S^c$. Define the unit vectors $u_{x_n}:=\frac{x_n-s_0}{\|x_n-s_0\|}$, $
u_{s_n}:=\frac{s_n-s_0}{\|s_n-s_0\|}$, and $u_{s_n'}:=\frac{s_n'-s_0}{\|s_n'-s_0\|}$, and consider $\zeta_n$ and $\xi_n$ the unit normal vectors to $S'$ at $s_n$ and $s_n'$, respectively, realized by an $r$-sphere, see Figure \ref{Fig2}.
\begin{figure}[h!]
\centering
\includegraphics[scale=0.25]{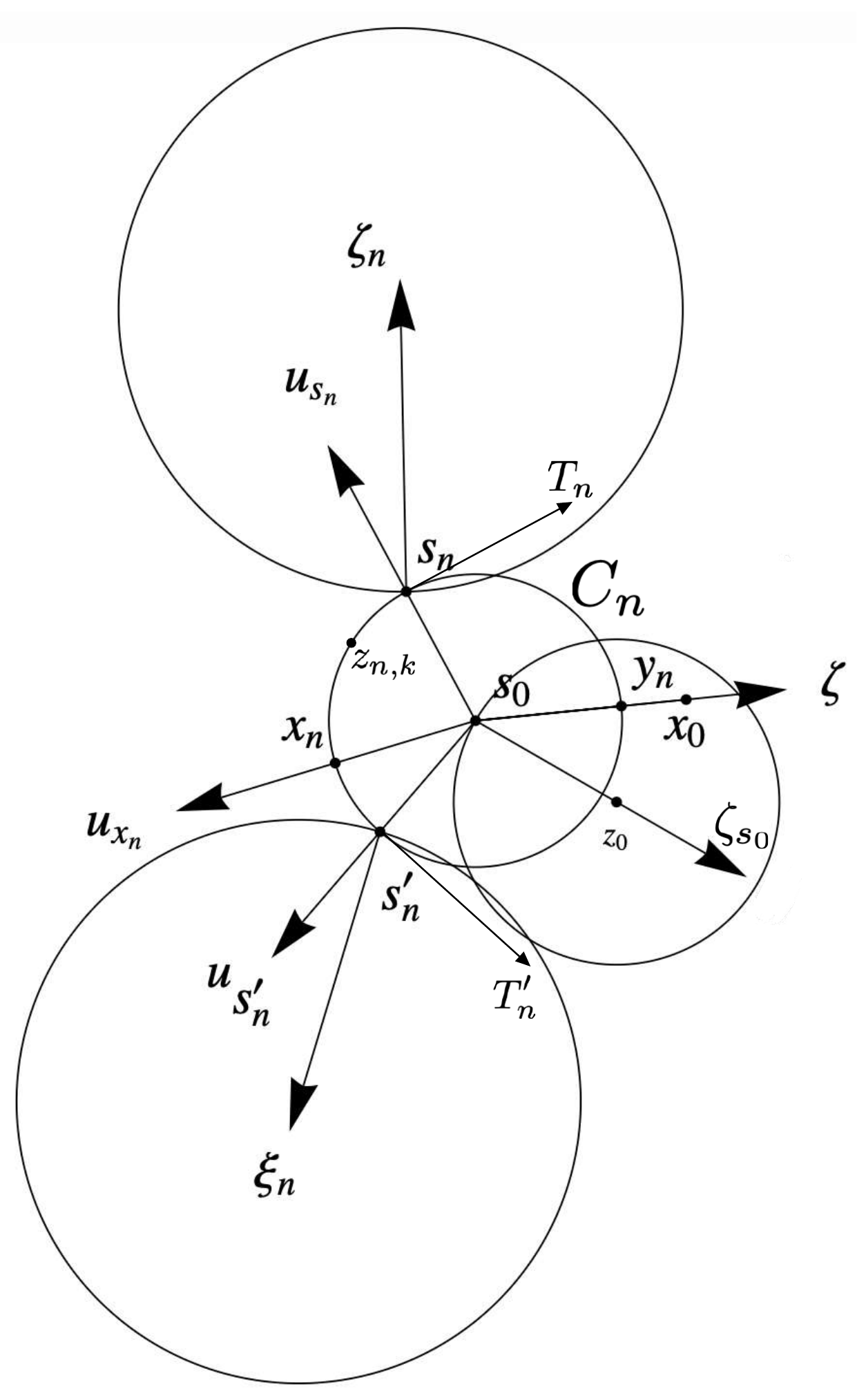}\;\;\;\;
\includegraphics[scale=1]{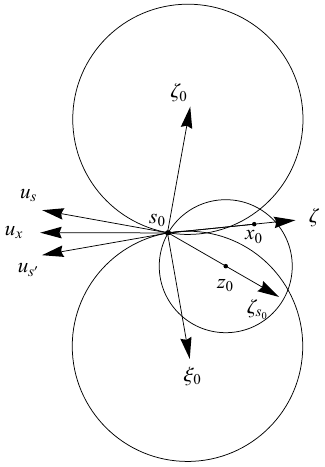}
\caption{\label{Fig2} Proof of Lemma~\ref{secondlem}}
\end{figure}
Since all these vectors are unit vectors, there exist subsequences, we do not relabel,  such that
\[
u_{x_n}\f u_x\;\hbox{unit},\;
u_{s_n}\f u_s\;\hbox{unit},\;
u_{s_n'}\f u_{s'}\;\hbox{unit},\;
\zeta_n\f \zeta_0\;\hbox{unit},\;\hbox{and}\;
\xi_n\f \xi_0 \;\hbox{unit}.
\]
Clearly,
\[
0\le \angle(\zeta,u_{s_n})< \angle(\zeta,u_{x_n})
< \angle(\zeta,u_{s_n'})< 2\pi,
\]
hence,
\begin{equation}\label{eq0}
0\le \angle(\zeta,u_s)\le \angle(\zeta,u_x)
\le \angle(\zeta,u_{s'})< 2\pi.
\end{equation}
Since $\zeta_n\in N_{S'}^P(s_n)$ is realized by an $r$-sphere we have $$\<\zeta_n,s-s_n\> \leq \frac{1}{2r}\|s-s_n\|^2,\;\;\forall s\in S'.$$
Taking $n\f\infty$ in this latter and using that $\|s_0-s_n\|=r_n\leq \frac{1}{n}$, we deduce that $$\<\zeta_0,s-s_0\> \leq \frac{1}{2r}\|s-s_0\|^2,\;\;\forall s\in S'.$$
This yields that $\zeta_0\in N_{S'}^P(s_0)$ is realized by an $r$-sphere. Similarly,  $\xi_0\in N_{S'}^P(s_0)$ is also realized by an $r$-sphere. Since $x_0\not\in S_{r^*}$, $r^*<r$, and $\overline{B}(s_0+r\zeta_0;r)\subset S$, we deduce that  $x_0\not \in \overline{B}(s_0+r\zeta_0;r)$. This latter gives  that $$\<\zeta,\zeta_0\><\frac{r_0}{2r}, $$
and hence \begin{equation}\label{eq1} \cos^{-1}\!\left(\frac{r_0}{2r}\right)
< \angle(\zeta,\zeta_0)
< 2\pi - \cos^{-1}\!\left(\frac{r_0}{2r}\right). \end{equation}
Similarly, $$\<\zeta,\xi_0\><\frac{r_0}{2r},$$ and hence \begin{equation}\label{eq2} \cos^{-1}\!\left(\frac{r_0}{2r}\right)
< \angle(\zeta,\xi_0)
< 2\pi - \cos^{-1}\!\left(\frac{r_0}{2r}\right). \end{equation}
From Lemma \ref{lem0}, we know that $S$ is regular closed, and hence $s_0\in\bdry S=\bdry S'\subset S'$. Add to this that $\zeta_n\in N_{S'}^P(s_n)$ is realized by an $r$-sphere, we deduce that $$\<\zeta_n,s_0-s_n\>\leq \frac{1}{2r}\|s_0-s_n\|^2,$$
and hence
\[
\langle -u_{s_n},\zeta_n\rangle\le \frac{r_n}{2r}.
\]
Passing to the limit gives
\[
\langle u_s,\zeta_0\rangle\ge 0.
\]
On the other hand, having $s_n\in\bdry S=\bdry S'\subset S'$ and $\zeta_0\in N_{S'}^P(s_0)$ is realized by an $r$-sphere, we get that $$\<\zeta_0,s_n-s_0\>\leq \frac{1}{2r}\|s_n-s_0\|^2,$$
and hence 
\[
\langle u_{s_n},\zeta_0\rangle\le \frac{r_n}{2r}.
\]
Passing to the limit gives
\[
\langle u_s,\zeta_0\rangle\le 0.
\]
Therefore, \begin{equation}\label{eq3}\<u_s,\zeta_0\>=0.\end{equation}
Similarly, \begin{equation*}\label{eq4}\<u_{s'},\xi_0\>=0.\end{equation*}
Fix $n$ sufficiently large. Since the open arc of $C_n$ joining $s_n$ to $s_n'$ and containing $x_n$
is contained in $S^c$, we may choose a sequence $(z_{n,k})_{k\ge 1}$ of points on this
arc such that
\[
z_{n,k}\to s_n\;\;\text{as }k\to\infty,\;\;\text{and}\;\;
z_{n,k}\notin S\;\;\hbox{for all}\;\;k.
\]
Set
\[
w_{n,k}:=\frac{z_{n,k}-s_n}{\|z_{n,k}-s_n\|}.
\]
Then $w_{n,k}\f -T_n$ as $k\f\infty$, where $T_n$ denotes the clockwise unit tangent to $C_n$
at $s_n$. Since $z_{n,k}\in S^c\subset S'$ and $\zeta_n\in N^P_{S'}(s_n)$ is realized by an $r$-sphere, we have
\[
\langle \zeta_n, z_{n,k}-s_n\rangle \le \frac{1}{2r}\|z_{n,k}-s_n\|^2.
\]
Dividing by $\|z_{n,k}-s_n\|>0$, we obtain
\[
\Big\langle \zeta_n, w_{n,k}\Big\rangle \le \frac{1}{2r}\|z_{n,k}-s_n\|.
\]
Letting $k\f\infty$ yields
\begin{equation*}
\langle \zeta_n,-T_n\rangle \le 0,
\;\;\text{that is,}\;\;
\langle \zeta_n,T_n\rangle \ge 0.
\end{equation*}
Assuming, up to extracting a further subsequence, that $T_n\to T_0$ unit as $n\f\infty$, we deduce that \begin{equation}\label{eq5}
\langle \zeta_0,T_0\rangle \ge 0.
\end{equation}
Since $u_{s_n}\f u_s$ and $T_n$ is obtained by rotating
$u_{s_n}$ by $\frac{\pi}{2}$ clockwise, we have
\begin{equation*}
\angle(T_0,u_s)=\frac{\pi}{2}.
\end{equation*}
Combining this latter with \eqref{eq3} and \eqref{eq5}, and since in $\mathbb{R}^2$ the direction orthogonal to $u_s$ is unique up to sign, we conclude that $\zeta_0=T_0$, and hence
\begin{equation*}\label{eq6}
\angle(\zeta_0,u_s)=\frac{\pi}{2}.
\end{equation*}
Similarly, using the symmetric construction at $s_n'$, we get
\begin{equation*}\label{eq7}
\angle(u_{s'},\xi_0)=\frac{\pi}{2}.
\end{equation*}
Therefore (see Figure~\ref{Fig2}),
\begin{equation}\label{eq9}
\angle(\zeta,\zeta_0)=\angle(\zeta,u_s)-\frac{\pi}{2}\;\;\hbox{and}\;\;\angle(\zeta,\xi_0)=\angle(\zeta,u_{s'})+\frac{\pi}{2}.
\end{equation}
Using \eqref{eq0} and \eqref{eq9}, we obtain
\[
\angle(\zeta,\zeta_0)
=\angle(\zeta,u_s)-\frac{\pi}{2}
\le \angle(\zeta,u_x)-\frac{\pi}{2},
\]
which together with \eqref{eq1} yields $(ii)$. Similarly,
\[
\angle(\zeta,\xi_0)
=\angle(\zeta,u_{s'})+\frac{\pi}{2}
\ge \angle(\zeta,u_x)+\frac{\pi}{2},
\]
and combining with \eqref{eq2} gives $(iii)$. From $(ii)$ and $(iii)$ we have
\[
\angle(\zeta,\zeta_0)\le \angle(\zeta,\xi_0),
\]
and since oriented angles take values in $[0,2\pi)$, it follows that
\[
\angle(\zeta_0,\xi_0)
=\angle(\zeta,\xi_0)-\angle(\zeta,\zeta_0).
\]
Using \eqref{eq1} and \eqref{eq2}, we deduce
\[
\angle(\zeta_0,\xi_0)
\ge \left(\angle(\zeta,u_x)+\frac{\pi}{2}\right)
-\left(\angle(\zeta,u_x)-\frac{\pi}{2}\right)
=\pi,
\]
and
\[
\angle(\zeta_0,\xi_0)
\le \left(2\pi-\cos^{-1}\!\left(\frac{r_0}{2r}\right)\right)
-\cos^{-1}\!\left(\frac{r_0}{2r}\right)
=2\pi-2\cos^{-1}\!\left(\frac{r_0}{2r}\right).
\]
This proves $(i)$. In particular, $\angle(\zeta_0,\xi_0)\ge \pi$ implies that 
$\zeta_0\neq \xi_0$, and therefore the boundary point $s_0$ is not regular.
\end{proof}
%
%

We are now ready to prove Theorem~\ref{mainthm}. We first outline the argument. We proceed by contradiction. Assume that there exists a point $x_0\in S$ that cannot be covered by a closed ball of radius $\ra$ contained in $S$. Using Lemma~\ref{secondlem}, we show that $x_0$ belongs to a closed ball contained in $S$ of radius strictly less than $\ra$ whose boundary contains three boundary points of $S$, namely $s_0$, $s'_0$, and $s''_0$, which are not regular and for which appropriate normal vectors can be chosen. Then, applying Lemma~\ref{geolem}, we prove that the sum of the angles of the triangle $s_0s'_0s''_0$ is strictly less than $\pi$, which yields a contradiction.

\begin{proofof}{Theorem~\ref{mainthm}} Assume by contradiction that $S$ is not the union of closed balls of radius $\ra$, that is, $S_{\ra}\subsetneq S$ where $S_{\ra}$ denotes the union of all the closed balls of radius $\ra$ contained in $S$. Then there exists $x_0\in S\cap\big(S_{\ra}\big)^c$. Let $s_0\in\proj_{\bdry S} (x_0)$, $r_0:=\|x_0-s_0\|$ and $\zeta:=\frac{x_0-s_0}{\|x_0-s_0\|}\in N_{S'}^P(s_0)$. Since $x_0\in\overline{B}(s_0+r_0\zeta)\subset S$ and $x_0\not\in S_{\ra}$, we deduce that $r_0<\ra$. Hence, applying Lemma \ref{secondlem} for $r^*:=\ra$, $x_0:=x_0$, $s_0:=s_0$, $\zeta:=\zeta$, $\zeta_{s_0}:=\zeta$ and $r':=r_0$, we get that $s_0$ is not regular, and  the existence of a unit vector $u_x$ and two distinct unit normal vectors $\zeta_0$ and $\xi_0$ to $S'$ at $s_0$, realized by an $r$-sphere, such that \begin{equation}\label{eq10} \pi\leq \angle(\zeta_0,\xi_0)\leq 2\pi-2\cos^{-1}\!\left(\frac{r_0}{2r}\right),\end{equation}
 \begin{equation}\label{eq11} \cos^{-1}\!\left(\frac{r_0}{2r}\right)\le \angle(\zeta,\zeta_0)\le \angle(\zeta,u_x) - \frac{\pi}{2},\;\hbox{and}\end{equation}
 \begin{equation}\label{eq12} \angle(\zeta,u_x) +\frac{\pi}{2}\leq \angle(\zeta,\xi_0)\leq 2\pi-\cos^{-1}\!\left(\frac{r_0}{2r}\right).\end{equation}
 This yields that  \begin{equation}\label{final}
\cos^{-1}\!\left(\frac{r_0}{2r}\right)
\le \angle(\zeta,\zeta_0)
\le \angle(\zeta,\xi_0)-\pi
\le \pi-\cos^{-1}\!\left(\frac{r_0}{2r}\right).
\end{equation}
Since $x_0\not\in S_{\ra}$ and $x_0\in \overline{B}(s_0+r_0\zeta)\subset \overline{B}\big(s_0+\ra\zeta;\ra\big)$,  we deduce that $$\overline{B}\left(s_0+\ra\zeta;\ra\right)\not\subset S.$$ Let $r_1$ be the largest $t>0$ such that $\overline{B}(s_0+t\zeta;t)\subset S$.
Clearly, we have \begin{equation}\label{eq13} r_0\leq r_1<\ra,\;\;\overline{B}(s_0+r_1\zeta;r_1)\subset S,\;\hbox{and}\;\zeta\in N_{S'}^P(s_0)\;\hbox{is realized by an}\;r_1\hbox{-sphere}.\end{equation}
For all $n\in\N$, we define the sequence $(x_n')_n $ by $$x'_n\in\overline{B}\left(s_0+\left(r_1+\frac{1}{n}\right)\zeta; r_1+\frac{1}{n}\right)\cap S^c.$$
As the sequence $(x_n')_n $ is bounded, it has a subsequence, we do not relabel, that converges to a point $s'_0$ satisfying \begin{equation} \label{eq13bis} s'_0\in \S(s_0+r_1\zeta;\zeta)\cap\clo (S^c)\subset S\cap\clo (S^c)=\bdry S.\end{equation}  Since $x'_n\in S^c\subset S'$ and $\zeta_0\in N_{S'}^P(s_0)$ is realized by an $r$-sphere, we have $$\left\<\zeta_0,\frac{x'_n-s_0}{\|x'_n-s_0\|}\right\>\leq \frac{1}{2r}\|x'_n-s_0\|.$$
Taking $n\f\infty$ in this latter, we deduce that \begin{equation}\label{eq14} \<v_0,\zeta_0\>\leq \frac{\|s'_0-s_0\|}{2r}, \end{equation}
where $v_0$ is the limit, when $n\f\infty$, of a subsequence not relabeled of $\left(\frac{x'_n-s_0}{\|x'_n-s_0\|}\right)_n$. Similarly, we have  \begin{equation}\label{eq15} \<v_0,\xi_0\>\leq \frac{\|s'_0-s_0\|}{2r}.\end{equation}
{\bf Claim 1:} $s'_0\not=s_0$. \vspace{0.2cm}\\
If not, then from \eqref{eq14} and \eqref{eq15} we deduce that \begin{equation} \label{eq16} \<v_0,\zeta_0\>\leq 0\;\;\hbox{and}\;\;\<v_0,\xi_0\>\leq 0.\end{equation}
From \eqref{eq13} and since $x'_n\in S^c\subset S'$, we get that $$\left\<\zeta,\frac{x'_n-s_0}{\|x'_n-s_0\|}\right\>\leq\frac{1}{2r_1}\|x'_n-s_0\|,$$
and hence, after taking $n\f\infty$, $$\<\zeta,v_0\>\leq0.$$
Add to this that $x'_n\in \overline{B}\left(s_0+\left(r_1+\frac{1}{n}\right)\zeta; r_1+\frac{1}{n}\right)$ which yields that $$\|x'_n-s_0\|\leq 2\left(r_1+\frac{1}{n}\right)\left\<\frac{x'_n-s_0}{\|x'_n-s_0\|},\zeta\right\>,$$ and hence, after taking $n\f\infty$, $$\<\zeta,v_0\>\geq0,$$ we deduce that $$\<\zeta,v_0\>=0.$$ This gives that $$\angle(v_0,\zeta)=\frac{\pi}{2}\;\;\hbox{or}\;\; \angle(\zeta,v_0)=\frac{\pi}{2}.$$
Combining this latter with \eqref{eq11} and \eqref{eq12}, we get that $$\<v_0,\zeta_0\>>0\;\;\hbox{or}\;\;\<v_0,\xi_0\>>0,$$
which contradicts \eqref{eq16}. 

Therefore, $$s'_0\not=s_0.$$
Let $x_1:=s_0+r_1\zeta$. Clearly we have $s_0\in\S(x_1;r_1)$. Moreover, from \eqref{eq13bis} we deduce that $s'_0\in\S(x_1;r_1)$, see Figure \ref{Fig4}. We define $$r'_0:=\|x_0-s'_0\|,\;\;\zeta':=\frac{x_0-s'_0}{\|x_0-s'_0\|},\;\hbox{and}\;\,\zeta_{s'_0}:=\frac{x_1-s'_0}{\|x_1-s'_0\|}. $$
\begin{figure}[h!]
\centering
\includegraphics[scale=0.45]{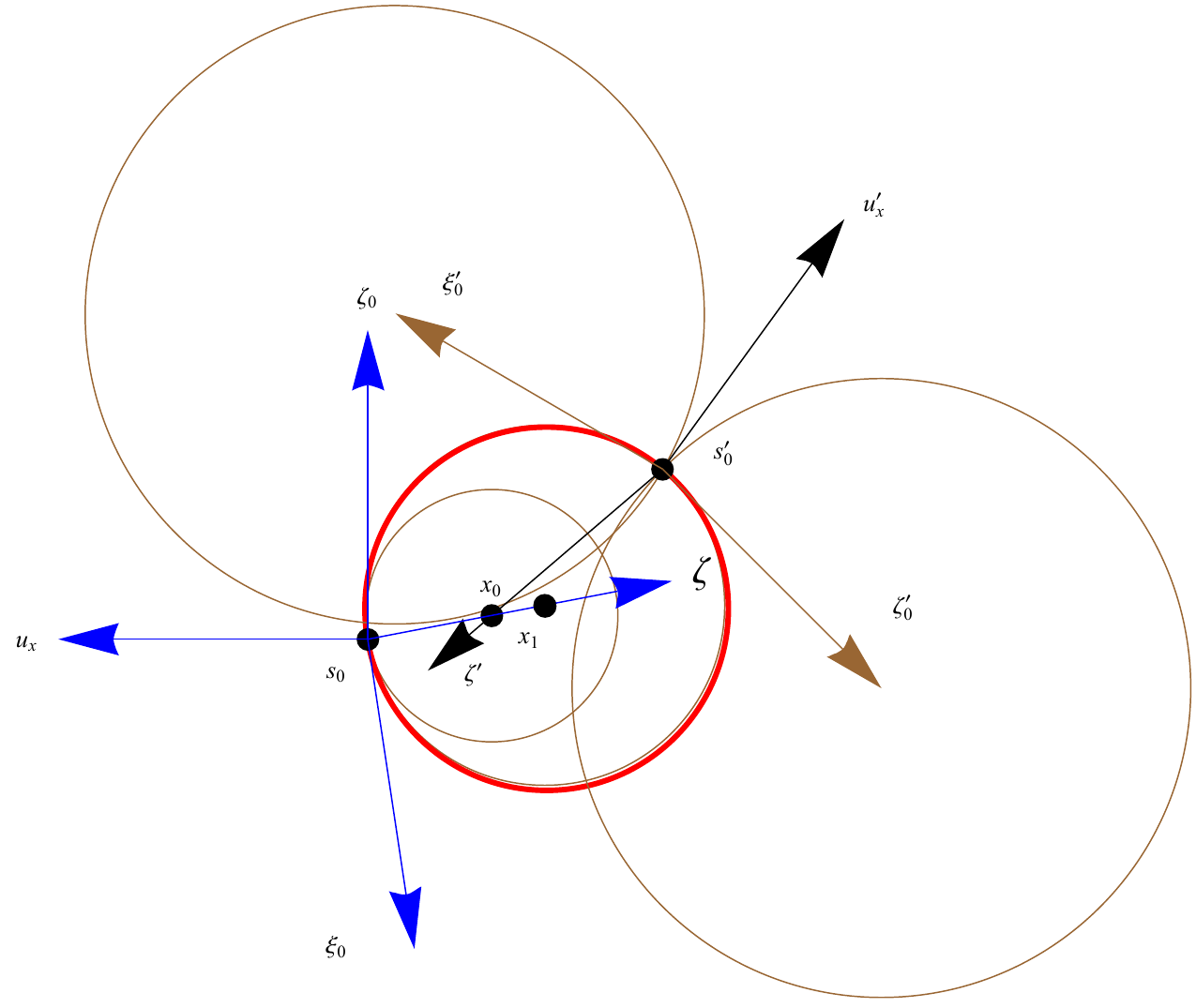}
\caption{\label{Fig4} The boundary point $s'_0$ is not regular} 
\end{figure}
Using \eqref{eq13}, we have $x_0\in \overline{B}(x_1;r_1)=\overline{B}(s'_0+r_1\zeta_{s'_0};r_1)\subset S$. This yields that $$ B(s'_0+r_1\zeta_{s'_0};r_1)\subset \inte S=(S')^c,$$
and hence, $\zeta_{s'_0}\in N_{S'}^P(s'_0)$ is realized by an $r_1$-sphere. Applying Lemma \ref{secondlem} for $r^*:=\ra$, $x_0:=x_0$, $s_0:=s'_0$, $\zeta:=\zeta'$, $\zeta_{s_0}=:\zeta_{s'_0}$ and $r':=r_1$, we get that $s'_0$ is not regular, and the existence of a unit vector $u'_x$ and two distinct unit normal vectors $\zeta'_0$ and $\xi'_0$ to $S'$ at $s'_0$, realized by an $r$-sphere, see Figure \ref{Fig4}, such that \begin{equation*}\label{eq17} \pi\leq \angle(\zeta'_0,\xi'_0)\leq 2\pi-2\cos^{-1}\!\left(\frac{r'_0}{2r}\right),\end{equation*}
 \begin{equation*}\label{eq18} \cos^{-1}\!\left(\frac{r'_0}{2r}\right)\le \angle(\zeta',\zeta'_0)\le \angle(\zeta',u'_x) - \frac{\pi}{2},\;\hbox{and}\end{equation*}
 \begin{equation*}\label{eq19} \angle(\zeta',u'_x) +\frac{\pi}{2}\leq \angle(\zeta',\xi'_0)\leq 2\pi-\cos^{-1}\!\left(\frac{r'_0}{2r}\right).\end{equation*}
 This yields that 
 \begin{equation}\label{finalbis}\cos^{-1}\!\left(\frac{r'_0}{2r}\right)
\le \angle(\zeta',\zeta'_0)
\le \angle(\zeta',\xi'_0)-\pi
\le \pi-\cos^{-1}\!\left(\frac{r'_0}{2r}\right). \end{equation}
Let $m:=\frac{s_0+s'_0}{2}$ and let $\zeta_1$ be the unit vector defined by $$\zeta_1:=\begin{cases} \frac{x_1-m}{\|x_1-m\|}, & \hbox{if}\;x_1\not=m,\vspace{0.1cm} \\\hbox{a unit vector perpendicular to}\;[s_0,s'_0], &  \hbox{if}\;x_1=m.  \end{cases}$$
As $\|x_1-s_0\|=\|x_1-s'_0\|=r_1$, we have that $\zeta_1$ is perpendicular to $[s_0,s'_0]$. Using  \eqref{eq13}, and the facts that $x_1= s_0+r_1\zeta$ and $x_0=s_0+r_0\zeta$ with $r_0\leq r_1$, we deduce that  $$x_0\in [s_0,x_1]\subset \overline{B}(x_1;r_1)= \overline{B}(m+\|x_1-m\|\zeta_1;r_1)=\overline{B}\left(m+\sqrt{r_1^2-\|s_0-m\|^2}\,\zeta_1;r_1\right)\subset S.$$
Since $\<s_0-m,\zeta_1\>=0$, we have $$\left\|s_0-m-\sqrt{\left(\ra\right)^2-\|s_0-m\|^2}\,\zeta_1\right\|^2=\left(\ra\right)^2,$$ which yields that \begin{eqnarray} \label{eq20} s_0&\in&\S\left(m+\sqrt{\left(\ra\right)^2-\|s_0-m\|^2}\,\zeta_1;\ra\right)\\&\subset& \overline{B}\left(m+\sqrt{\left(\ra\right)^2-\|s_0-m\|^2}\,\zeta_1;\ra\right).\nonumber \end{eqnarray} Moreover,
\begin{eqnarray} \nonumber \left\|x_1- m-\sqrt{\left(\ra\right)^2-\|s_0-m\|^2}\,\zeta_1\right\|&=& \nonumber \left\|\|x_1- m\|\zeta_1-\sqrt{\left(\ra\right)^2-\|s_0-m\|^2}\,\zeta_1\right\|\\ &=&\nonumber \left|\|x_1-m\|-\sqrt{\left(\ra\right)^2-\|s_0-m\|^2}\right| \\&=& \nonumber\sqrt{\left(\ra\right)^2-\|s_0-m\|^2} - \sqrt{r_1^2-\|s_0-m\|^2} \\ &\leq& \ra.\nonumber  \end{eqnarray}
Then $x_1\in\overline{B}\left(m+\sqrt{\left(\ra\right)^2-\|s_0-m\|^2}\,\zeta_1;\ra\right)$. Hence, $$x_0\in [s_0,x_1]\subset \overline{B}\left(m+\sqrt{\left(\ra\right)^2-\|s_0-m\|^2}\,\zeta_1;\ra\right),$$
which yields, since $x_0\not\in S_{\ra}$, that $$\overline{B}\left(m+\sqrt{\left(\ra\right)^2-\|s_0-m\|^2}\,\zeta_1;\ra\right)\not\subset S. $$
Let $r_2$ be the largest $t>0$ such that $$\overline{B}\left(m+\sqrt{t^2-\|s_0-m\|^2}\,\zeta_1;t\right)\subset S.$$
Clearly we have $$r_1\leq r_2<\ra\;\;\hbox{and}\;\;\overline{B}\left(m+\sqrt{r_2^2-\|s_0-m\|^2}\,\zeta_1;r_2\right)\subset S.$$ Moreover, for $x_2:=m+\sqrt{r_2^2-\|s_0-m\|^2}\,\zeta_1$, we have, using arguments similar to whose used to prove \eqref{eq20}, see Figure \ref{Fig5}, that $$s_0\in\S (x_2;r_2)\subset \overline{B}(x_2;r_2)\subset S\;\;\hbox{and}\;\;s'_0\in\S (x_2;r_2)\subset \overline{B}(x_2;r_2)\subset S.$$
\begin{figure}[h!]
\centering
\includegraphics[scale=0.47]{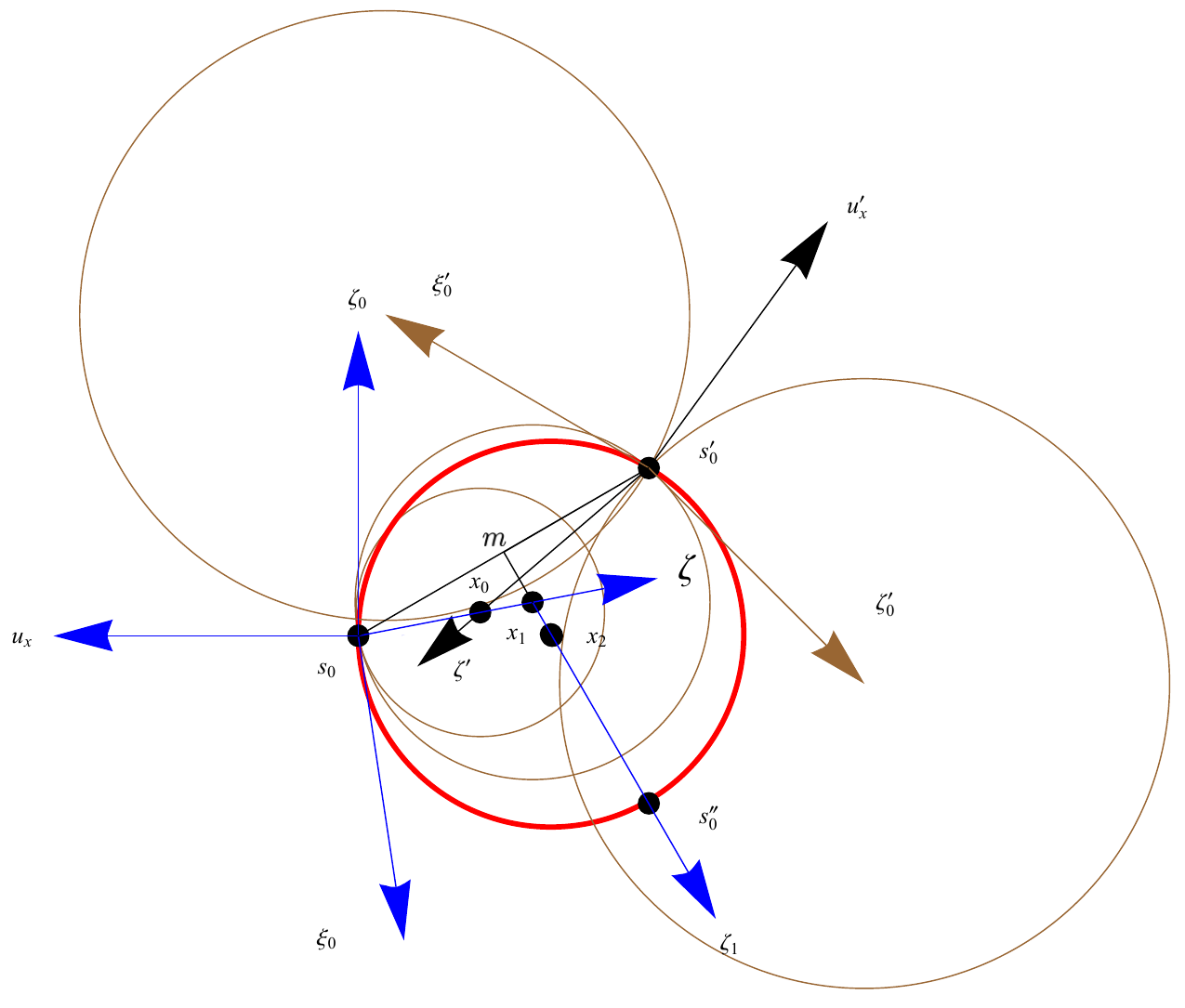}
\caption{\label{Fig5} The ball $\overline{B}(x_2;r_2)$} 
\end{figure}
We proceed to prove that  $\overline{B}(x_2;r_2)$ contains a third boundary point $s''_0$ of $S$ which is not regular. For all $n\in\N$, we define the sequence $(x_n'')_n $ by $$x''_n\in\overline{B}\left(m+\sqrt{\left(r_2+\frac{1}{n}\right)^2-\|s_0-m\|^2}\,\zeta_1;r_2+\frac{1}{n}\right)\cap S^c.$$
As the sequence $(x_n'')_n $ is bounded, it has a subsequence, we do not relabel, that converges to a point $s''_0$ satisfying \begin{equation*} \label{eq21} s''_0\in \S(x_2;r_2)\cap\clo (S^c)\subset \overline{B}(x_2;r_2)\cap\clo (S^c)\subset  S\cap\clo (S^c)=\bdry S.\end{equation*} 
{\bf Claim 2:} $s''_0\not=s_0$ and $s''_0\not=s'_0$. \vspace{0.2cm}\\
It suffices to prove that $s''_0\neq s_0$, since $s''_0\neq s'_0$ follows by the same argument. Since $x''_n\in S^c\subset S'$ and $\zeta_0\in N_{S'}^P(s_0)$ is realized by an $r$-sphere, we have $$\left\<\zeta_0,\frac{x''_n-s_0}{\|x''_n-s_0\|}\right\>\leq \frac{1}{2r}\|x''_n-s_0\|.$$
Taking $n\f\infty$ in this latter, we deduce that \begin{equation}\label{eq14bis} \<v''_0,\zeta_0\>\leq \frac{\|s''_0-s_0\|}{2r}, \end{equation}
where $v''_0$ is the limit, when $n\f\infty$, of a subsequence not relabeled of $\left(\frac{x''_n-s_0}{\|x''_n-s_0\|}\right)_n$. Similarly, we have  \begin{equation}\label{eq15bis} \<v''_0,\xi_0\>\leq \frac{\|s''_0-s_0\|}{2r}.\end{equation}
If $s''_0=s_0$, then from \eqref{eq14bis} and \eqref{eq15bis} we deduce that \begin{equation} \label{eq16bis} \<v''_0,\zeta_0\>\leq 0\;\;\hbox{and}\;\;\<v''_0,\xi_0\>\leq 0.\end{equation}
For $u_0:=\frac{x_2-s_0}{\|x_2-s_0\|}$, we have $\bar{B}(s_0+r_2u_0;r_2)=\bar{B}(x_2;r_2)\subset S$. Then $u_0\in N_S^P(s_0)$ and is realized by an $r_2$-sphere. Since $x''_n\in S^c\subset S'$, we get that $$\left\<u_0,\frac{x''_n-s_0}{\|x''_n-s_0\|}\right\>\leq\frac{1}{2r_2}\|x''_n-s_0\|,$$
and hence, after taking $n\f\infty$, \begin{equation}\label{end1}  \<u_0,v''_0\>\leq0.\end{equation}
On the other hand, we have $$x''_n\in \overline{B}\left(m+\sqrt{\left(r_2+\frac{1}{n}\right)^2-\|s_0-m\|^2}\,\zeta_1;r_2+\frac{1}{n}\right)=\bar{B}\left(s_0+\left(r_2+\frac{1}{n}\right)u_n;r_2+\frac{1}{n}\right),$$ where $u_n:=\frac{x_{2,n}-s_0}{\|x_{2,n}-s_0\|}$ and $x_{2,n}:= m+\sqrt{\left(r_2+\frac{1}{n}\right)^2-\|s_0-m\|^2}\,\zeta_1$, which yields that  $$\|x''_n-s_0\|\leq 2\left(r_2+\frac{1}{n}\right)\left\<\frac{x''_n-s_0}{\|x''_n-s_0\|},u_n\right\>,$$ and hence, after taking $n\f\infty$, $$\<u_0,v''_0\>\geq0.$$ Combining this latter with \eqref{end1}, we deduce that $$\<u_0,v''_0\>=0.$$ This gives that \begin{equation} \label{end2} \angle(v''_0,u_0)=\frac{\pi}{2}\;\;\hbox{or}\;\; \angle(u_0,v''_0)=\frac{\pi}{2}.\end{equation}
We have \begin{eqnarray} \nonumber \left\|x_1- m-\sqrt{r_2^2-\|s_0-m\|^2}\,\zeta_1\right\|&=& \nonumber \left\|\|x_1- m\|\zeta_1-\sqrt{r_2^2-\|s_0-m\|^2}\,\zeta_1\right\|\\ &=&\nonumber \left|\|x_1-m\|-\sqrt{r_2^2-\|s_0-m\|^2}\right| \\&=& \nonumber\sqrt{r_2^2-\|s_0-m\|^2} - \sqrt{r_1^2-\|s_0-m\|^2} \\ &\leq& r_2.\nonumber  \end{eqnarray}
Then $x_1\in\overline{B}(x_2;r_2)$, and hence \begin{equation} \label{eq22} x_0\in [s_0,x_1]\subset  \overline{B}(x_2;r_2)=\bar{B}(s_0+r_2u_0;r_2)\subset S.\end{equation}
This yields that $\|x_0-s_0\|\leq 2r_2\<\zeta,u_0\> $, and hence, $\<\zeta,u_0\>\geq \frac{r_0}{2r_2}>\frac{r_0}{2r}$. Combining this latter with \eqref{eq11}-\eqref{eq12}, we conclude that \begin{equation} \label{u_01}
0 \le \angle(u_0,\zeta) \le \cos^{-1}\!\left(\frac{r_0}{2r_2}\right)
< \cos^{-1}\!\left(\frac{r_0}{2r}\right)
\le \angle(\xi_0,\zeta),
\end{equation}
or
 \begin{equation}\label{u_02}
0 \le \angle(\zeta,u_0) \le \cos^{-1}\!\left(\frac{r_0}{2r_2}\right)
< \cos^{-1}\!\left(\frac{r_0}{2r}\right)
\le \angle(\zeta,\zeta_0).
\end{equation}
In particular,
\begin{equation}\label{angle-u0}
u_0 \neq \zeta_0\;\; \text{and}\;\;
u_0 \neq \xi_0.
\end{equation}
Without loss of generality, we assume that \eqref{u_01} holds, as the case \eqref{u_02} leads to same results from similar arguments. \vspace{0.2cm}\\
\underline{Case 1}: $0 \le \angle(\xi_0,v''_0) < \angle(\xi_0,\zeta_0)$ (see Figure \ref{Fig8123}).\vspace{0.2cm}\\
Then, using \eqref{eq10} and \eqref{eq16bis}, we get that 
\[
\pi\geq \angle(\xi_0,\zeta_0)
= \angle(\xi_0,v''_0)+\angle(v''_0,\zeta_0)
\ge \frac{\pi}{2}+\frac{\pi}{2}
= \pi,
\]
which implies
\[
\angle(\xi_0,\zeta_0)=\pi,\;\;\angle(\xi_0,v''_0)=\frac{\pi}{2},\;\;\hbox{and}\;\;
\angle(v''_0,\zeta_0)=\frac{\pi}{2}.
\]
Since $\langle v''_0,u_0\rangle=0$, this yields that $u_0 \in \{\zeta_0,\xi_0\},$ which contradicts \eqref{angle-u0}.\vspace{0.2cm}\\
\underline{Case 2}: $0 \le \angle(\zeta_0,v''_0) < \angle(\zeta_0,\xi_0)$ (see Figure \ref{Fig8123}).\vspace{0.2cm}\\
In view of \eqref{end2}, we distinguish the following two subcases.\vspace{0.2cm}\\
\underline{Case 2.1}: $\angle(v''_0,u_0)=\frac{\pi}{2}$.\vspace{0.2cm}\\
\begin{figure}[t]
\centering
\includegraphics[scale=0.237]{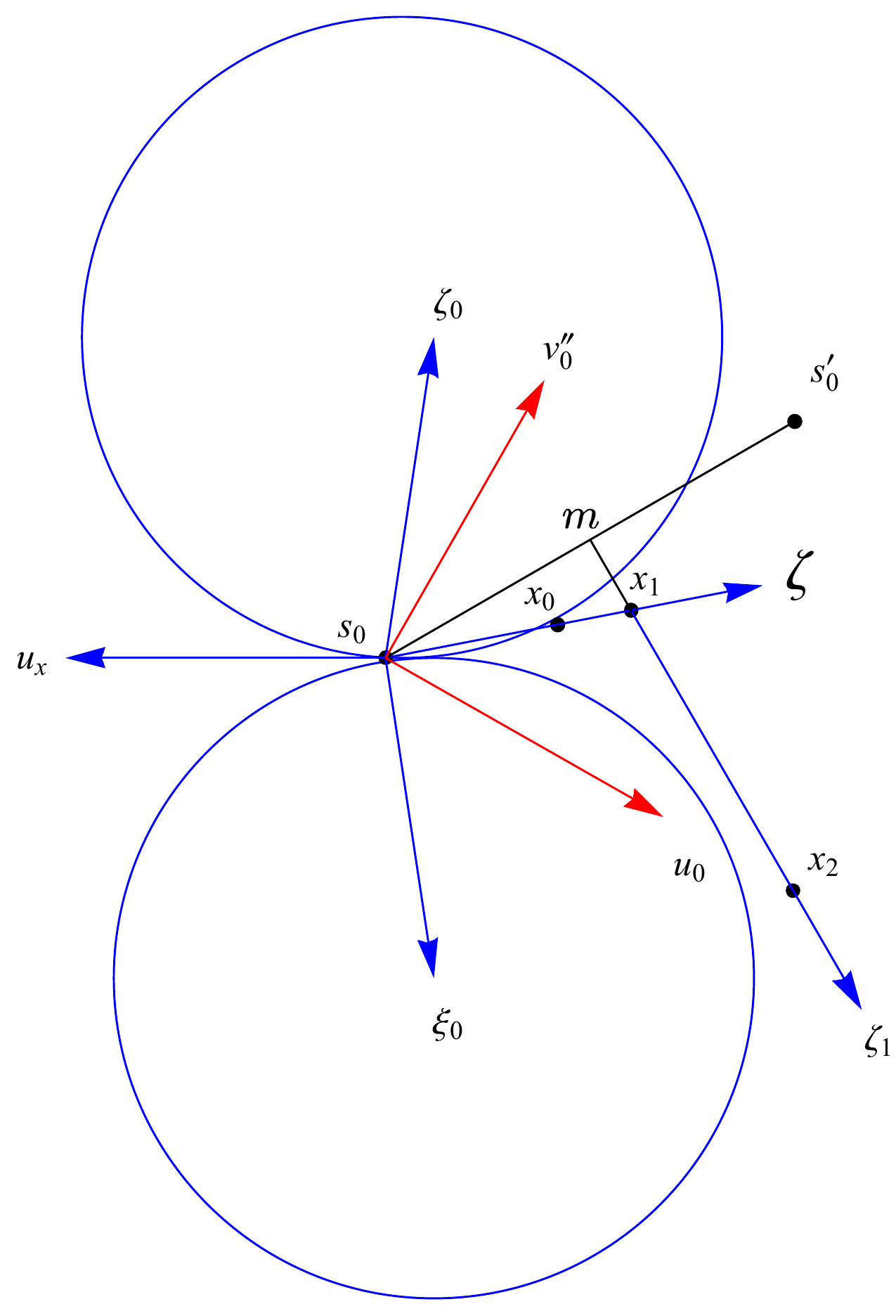}
\includegraphics[scale=0.237]{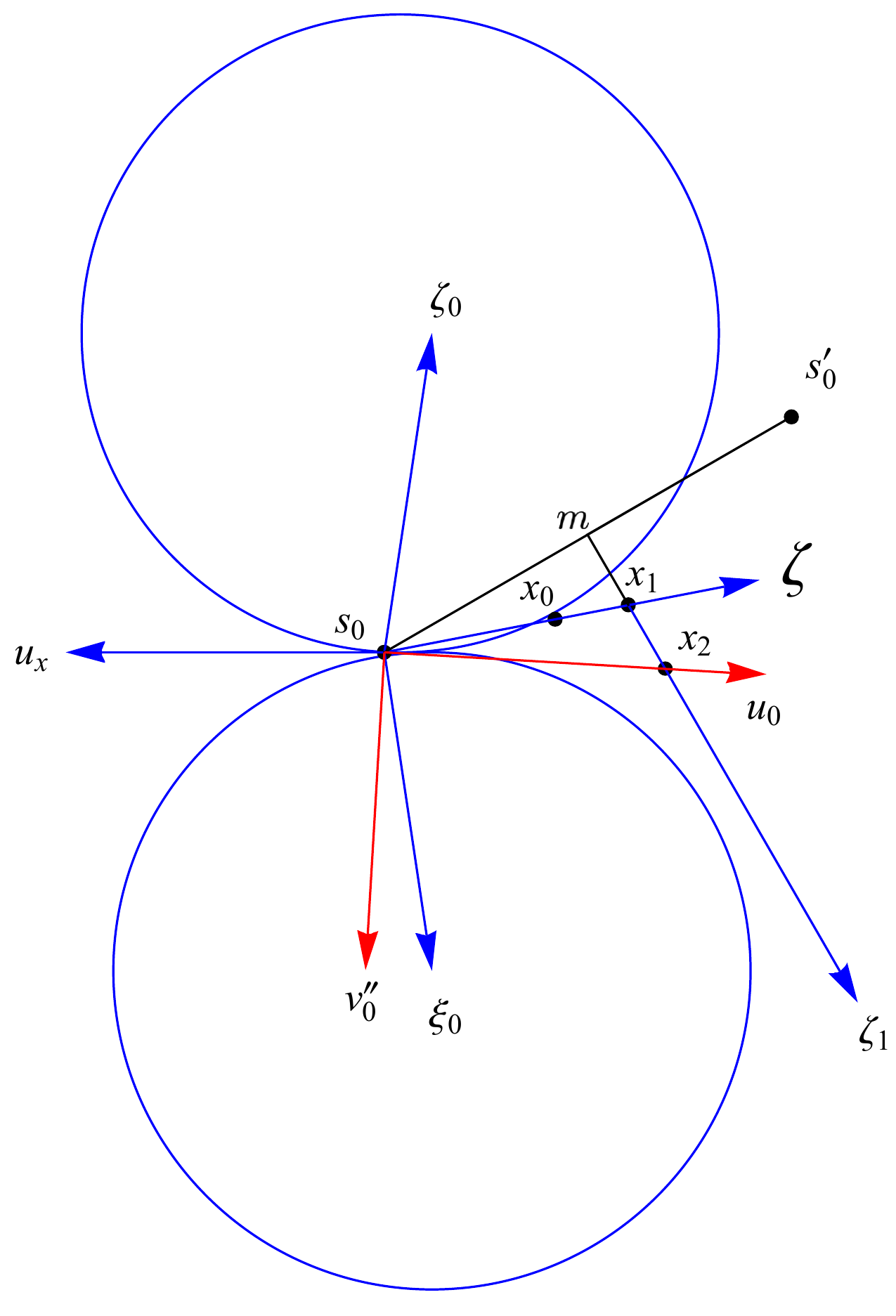}
\includegraphics[scale=0.25]{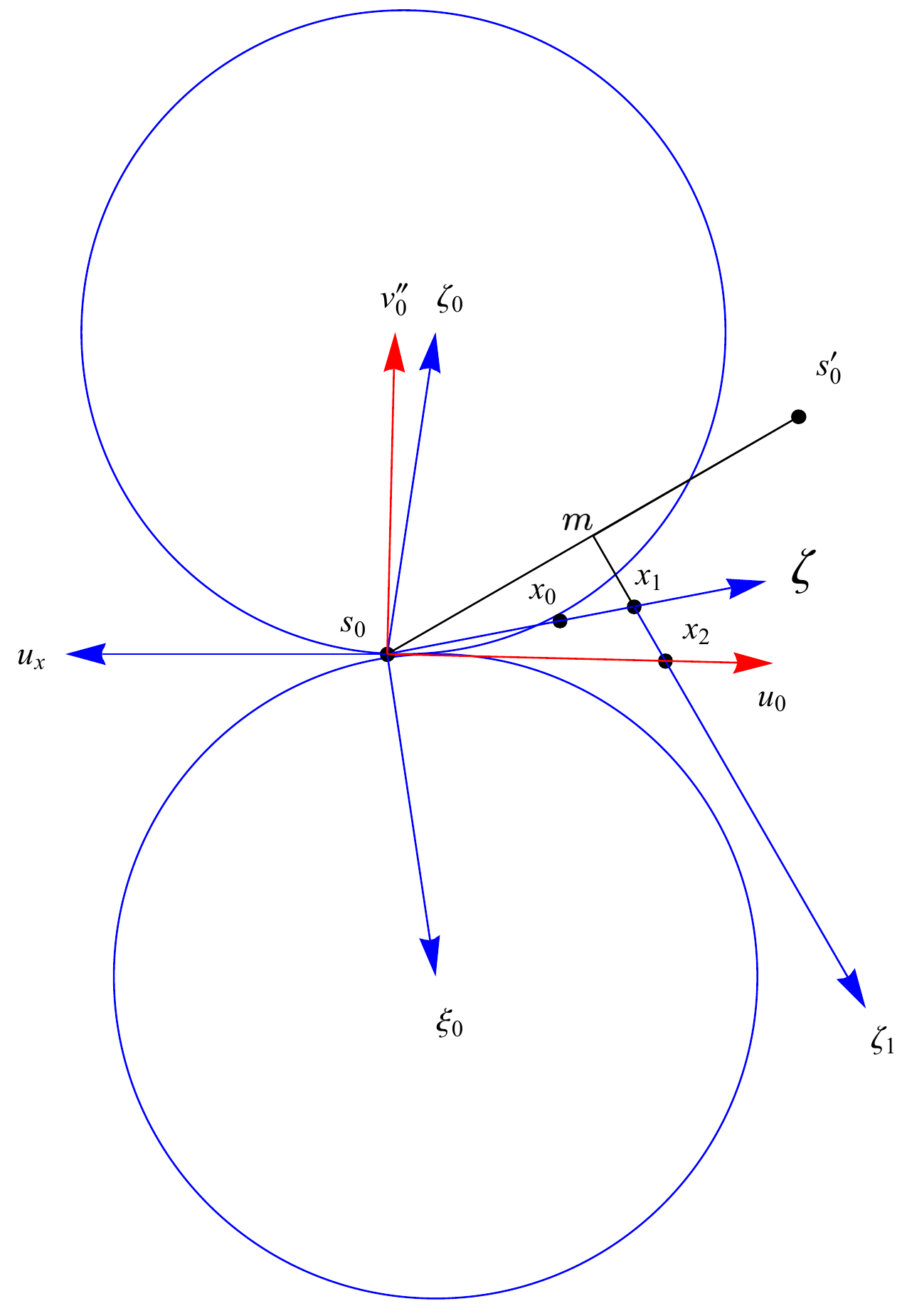}
\caption{\label{Fig8123} $s''_0\not=s_0$} 
\end{figure}
Then, 
\begin{eqnarray*}0\leq\angle(v''_0,\xi_0)=\angle(v''_0,u_0) + \angle(u_0,\zeta)-\angle(\xi_0,\zeta)&=&\frac{\pi}{2}+\angle(u_0,\zeta)-\angle(\xi_0,\zeta)\\ &\overset{\eqref{u_01}}{\leq}& \frac{\pi}{2} +  \cos^{-1}\!\left(\frac{r_0}{2r_2}\right)
-\cos^{-1}\!\left(\frac{r_0}{2r}\right)\\&<&\frac{\pi}{2}.\end{eqnarray*}
This gives that $\langle v''_0,\xi_0\rangle>0$ which contradicts \eqref{eq16bis}.\vspace{0.2cm}\\
\underline{Case 2.2}: $\angle(u_0,v''_0)=\frac{\pi}{2}$.\vspace{0.2cm}\\
Then,
\begin{eqnarray*}0\leq\angle(\zeta_0,v''_0)=\angle(u_0,v''_0) - \angle(u_0,\zeta)-\angle(\zeta,\zeta_0)&=&\frac{\pi}{2}-\angle(u_0,\zeta)-\angle(\zeta,\zeta_0)\\ &\leq& \frac{\pi}{2} -\angle(\zeta,\zeta_0) \\ &\overset{\eqref{eq11}}{\leq}& \frac{\pi}{2}-\cos^{-1}\!\left(\frac{r_0}{2r}\right)\\&<&\frac{\pi}{2}.\end{eqnarray*}
This gives that $\langle v''_0,\zeta_0\rangle>0$ which again contradicts \eqref{eq16bis}.

Therefore, $$s''_0\not=s_0\;\,\hbox{and}\;\,s''_0\not=s'_0.$$
Let $r''_0:=\|x_0-s''_0\|$, $\zeta'':=\frac{x_0-s''_0}{\|x_0-s''_0\|}$ and $\zeta_{s''_0}:=\frac{x_2-s''_0}{\|x_2-s''_0\|}$. From \eqref{eq22}, we have $$x_0\in \overline{B}(x_2;r_2)=\overline{B}(s''_0+r_2\zeta_{s''_0};r_2)\subset S.$$
This yields that $\zeta_{s''_0}\in N_S^P(s''_0)$ is realized by an $r_2$-sphere. Applying Lemma \ref{secondlem} for $r^*:=\ra$, $x_0:=x_0$, $s_0:=s''_0$, $\zeta:=\zeta''$, $\zeta_{s_0}=:\zeta_{s''_0}$ and $r':=r_2$, we get that $s''_0$ is not regular, and the existence of a unit vector $u''_x$ and two distinct unit normal vectors $\zeta''_0$ and $\xi''_0$ to $S'$ at $s''_0$, realized by an $r$-sphere, see Figure \ref{Fig6}, such that \begin{equation*}\label{eq17bis} \pi\leq \angle(\zeta''_0,\xi''_0)\leq 2\pi-2\cos^{-1}\!\left(\frac{r''_0}{2r}\right),\end{equation*}
 \begin{equation*}\label{eq18bis} \cos^{-1}\!\left(\frac{r''_0}{2r}\right)\le \angle(\zeta'',\zeta''_0)\le \angle(\zeta'',u''_x) - \frac{\pi}{2},\;\hbox{and}\end{equation*}
 \begin{equation*}\label{eq19bis} \angle(\zeta'',u''_x) +\frac{\pi}{2}\leq \angle(\zeta'',\xi''_0)\leq 2\pi-\cos^{-1}\!\left(\frac{r''_0}{2r}\right).\end{equation*}
This yields that
\begin{equation}\label{finalbisbis}\cos^{-1}\!\left(\frac{r''_0}{2r}\right)
\le \angle(\zeta'',\zeta''_0)
\le \angle(\zeta'',\xi''_0)-\pi
\le \pi-\cos^{-1}\!\left(\frac{r''_0}{2r}\right). \end{equation}
 \begin{figure}[h!]
\centering
\includegraphics[scale=0.4]{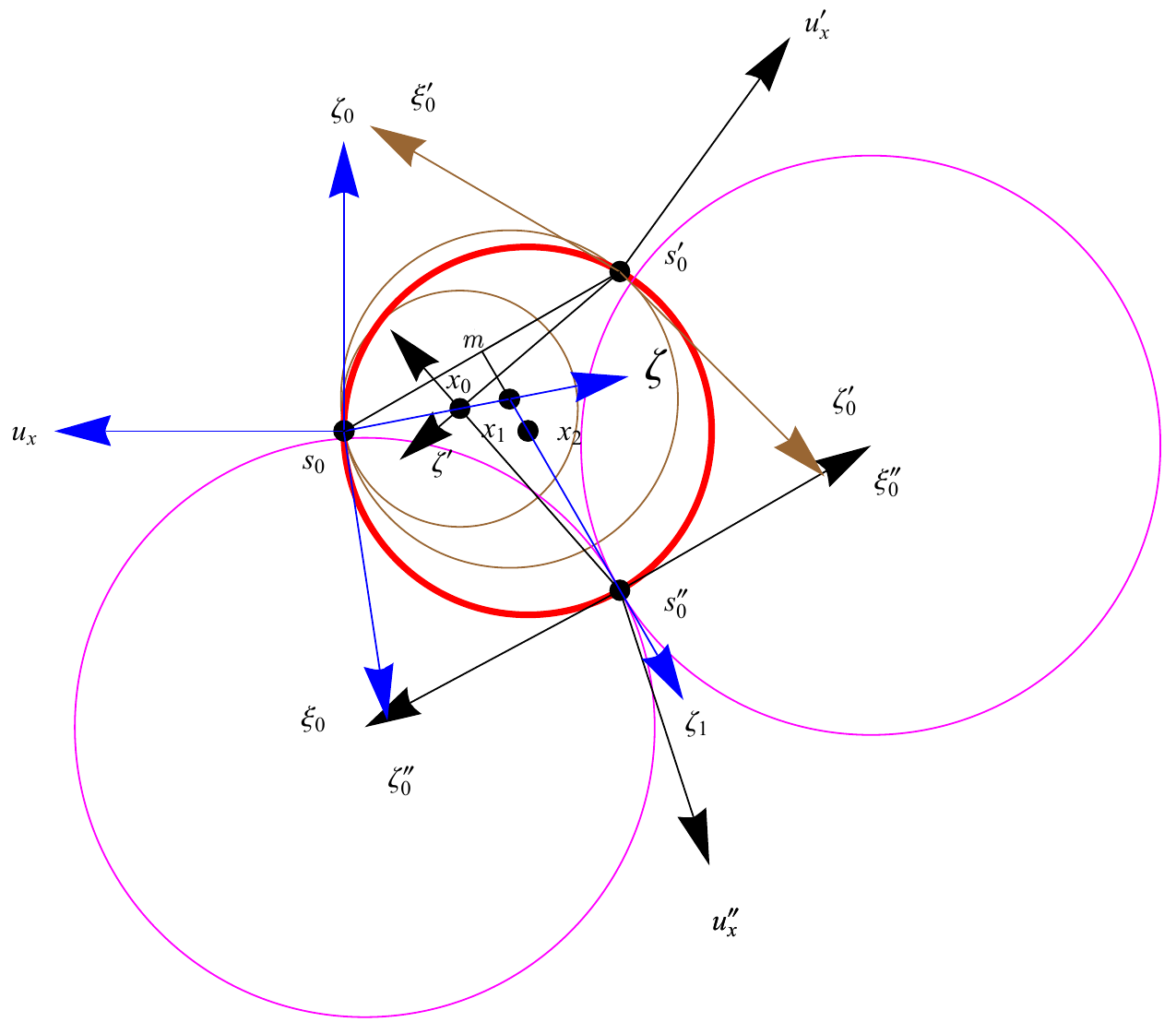}
\caption{\label{Fig6} The boundary point $s''_0$ is not regular} 
\end{figure}
Using \eqref{final}, and taking $O:=s_0$, $A:=x_0$, $B:=s_0+r\xi_0$, $C:=s_0+r\zeta_0$, $N:=s'_0$, and $M:=s''_0$, the assumptions of Lemma~\ref{geolem} are satisfied (see Figure~\ref{Fig7}). Hence,
\[
\angle s''_0 s_0 s'_0 < \frac{\pi}{3}.
\]
\begin{figure}[t]
\centering
\includegraphics[scale=1.5]{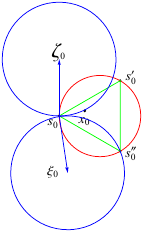}
\includegraphics[scale=1.2]{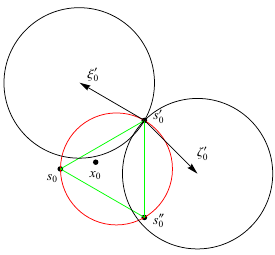}
\includegraphics[scale=1.2]{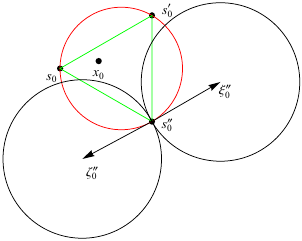}
\caption{\label{Fig7} The angles $\angle s''_0 s_0 s'_0$, $\angle s_0 s'_0 s''_0$ and $\angle s'_0 s''_0 s_0$} 
\end{figure}
Similarly, using \eqref{finalbis}, and taking $O:=s'_0$, $A:=x_0$, $B:=s'_0+r\xi'_0$, $C:=s'_0+r\zeta'_0$, $N:=s_0$ and $M:=s''_0$, the assumptions of Lemma~\ref{geolem} are satisfied (see Figure~\ref{Fig7}). Hence,
\[
\angle s_0 s'_0 s''_0 < \frac{\pi}{3}.
\]
Similarly, using \eqref{finalbisbis}, and taking $O:=s''_0$, $A:=x_0$, $B:=s''_0+r\xi''_0$, $C:=s''_0+r\zeta''_0$, $N:=s_0$ and $M:=s'_0$, the assumptions of Lemma~\ref{geolem} are satisfied (see Figure~\ref{Fig7}). Hence,
\[
\angle s'_0 s''_0 s_0 < \frac{\pi}{3}.
\]
Summing the three inequalities, we obtain
\[
\angle s''_0 s_0 s'_0
+\angle s_0 s'_0 s''_0 
+\angle s'_0 s''_0 s_0 
< \pi,
\]
which contradicts the fact that the sum of the angles of the triangle $s_0 s'_0 s''_0$ is equal to $\pi$. The proof of Theorem~\ref{mainthm} is completed. \end{proofof}

\begin{remark} The proof shows that the contradiction arises from an extremal geometric configuration, corresponding to the case where the angle estimates in Lemma~\ref{secondlem} are saturated. In this situation, one has $\xi_0 = -\zeta_0$ (and similarly at the other contact points), and the three boundary points lie on the spheres realizing the corresponding normal directions. If $r_0<\ra$, then at least one of the boundary points belongs to the interior of one of these balls, which contradicts the construction.
\end{remark}

\section{Conclusion and perspectives} \label{sectionfinal}

The proof of Theorem~\ref{mainthm} relies on a geometric mechanism that admits a natural interpretation in higher dimension. 
The first step of the argument consists in showing that if there exists a point
$x_0\in S\cap \big(S_{\ra}\big)^c$, then $x_0$ belongs to a closed ball contained in $S$
whose boundary contains several boundary points of $S$. In the planar case, this construction yields a closed ball whose boundary contains three boundary points of $S$ that are not regular. The contradiction is then obtained from a purely two-dimensional geometric argument based on angle estimates.

In $\mathbb{R}^n$, the same construction suggests that, under the same assumption,
one should obtain a closed ball contained in $S$ whose boundary contains at least
$n+1$ boundary points of $S$ that are not regular. This corresponds to the minimal number of contact points required to determine a sphere in dimension $n$. Therefore, the first part of the argument is not specific to the planar case and could potentially be extended to arbitrary dimension.

The main difficulty arises in the final step of the proof. In the plane, the contradiction follows from the fact that the sum of the angles of a triangle is equal to $\pi$, together with sharp local angle estimates. This argument has no direct analogue in higher dimension, where the geometry of normal cones is more complex and cannot be reduced to a one-dimensional angular structure. Thus, extending Theorem~\ref{mainthm} to $\mathbb{R}^n$, $n\ge 3$, requires identifying a higher–dimensional substitute for the planar angle argument used in the last step of the proof.

\end{document}